\documentclass{article}
\usepackage[utf8]{inputenc}
\usepackage[english]{babel}
\usepackage{amsmath}
\usepackage{amsthm}
\usepackage{tikz}
\usepackage{amsfonts}
\usepackage{amssymb}
\usepackage{graphicx}
\usepackage{xcolor}
\usepackage{bbm}
\usepackage{csquotes}
\usepackage{latexsym}
\usepackage{keyval}
\usepackage[backend=biber, style=alphabetic]{biblatex}
\newtheorem{thm}{Theorem}[section]
\newtheorem*{thm212}{Theorem \ref{interab}}
\newtheorem*{thm328}{Theorem \ref{Main}}
\newtheorem{lem}[thm]{Lemma}

\newtheorem{prop}[thm]{Proposition}
\newtheorem{coro}[thm]{Corollary}
\newtheorem{conj}[thm]{Conjecture}
\theoremstyle{definition}
\newtheorem{defn}[thm]{Definition}

\newtheorem{rem}[thm]{Remark}

\newtheorem{fact}[thm]{Fact}

\newcommand{\trdeg}{\textnormal{trdeg}}

\newcommand{\End}{\textnormal{End}}
\newcommand{\codim}{\textnormal{codim}}

\renewcommand{\exp}{\textnormal{exp}}

\renewcommand{\o}[1]{\overline{ #1 }}

\newcommand{\im}{\textnormal{im}}

\renewcommand{\sl}{\textnormal{SL}}
\newcommand{\gl}{\textnormal{GL}}

\renewcommand{\epsilon}{\varepsilon}
\renewcommand{\phi}{\varphi}
\newcommand{\so}{\textnormal{SO}}

\newcommand{\mat}[4]{ \left(\begin{matrix} #1 & #2 \\ #3 & #4\end{matrix} \right) }

\newcommand{\vect}[2]{ \left(\begin{matrix} #1 \\ #2 \end{matrix} \right) }
\newcommand{\res}{\textnormal{res}}
\renewcommand{\Re}{\textnormal{Re}}
\renewcommand{\Im}{\textnormal{Im}}
\title{Solving Systems of Equations of Raising-to-Powers Type}
\author{Francesco Paolo Gallinaro}

\begin{document}
\maketitle

\begin{abstract}
We address special cases of the analogues of the exponential algebraic closedness conjecture relative to the exponential maps of semiabelian varieties and to the modular $j$ function. In particular, we show that the graph of the exponential of an abelian variety intersects products of rotund varieties in which the subvariety of the domain is a sufficiently generic linear subspace, and that the graph of $j$ intersects products of free broad varieties in which the subvariety of the domain is a M\"obius subvariety.
\end{abstract}

\section{Introduction}

Zilber's study of the model theory of complex exponentiation has originated a lot of work around systems of exponential polynomial equation. The main conjecture, known as \textit{exponential algebraic closedness} and first stated in \cite{Zil05}, predicts sufficient conditions for an algebraic subvariety of $\mathbb{C}^n \times (\mathbb{C}^\times)^n$ to intersect the graph of the exponential function. It is related to Schanuel's conjecture, a famous open problem in transcendental number theory.

\begin{conj}[Schanuel]
Let $z_1,\dots,z_n \in \mathbb{C}^n$ be $\mathbb{Q}$-linearly independent complex numbers. Then $$\trdeg(z_1,\dots,z_n,e^{z_1},\dots,e^{z_n}) \geq n.$$
\end{conj}

Schanuel's conjecture predicts that algebraic varieties of low dimension defined over $\mathbb{Q}$ do not intersect the graph of the exponential function, unless they are forced to do so by the fact that $\exp$ is a group homomorphism. In the 1970s, Ax \cite{Ax72, Ax72b} proved a function field version of this conjecture, establishing that the dimension of the irreducible analytic components of the intersections of complex algebraic varieties with the graph of the exponential are governed by $\mathbb{Q}$-linear relations on the domain and by multiplicative relations on the codomain.

The exponential algebraic closedness conjecture can be seen as a dual to Ax's theorem: it predicts that all algebraic varieties which ``should'' intersect the graph of the exponential, based on the $\mathbb{Q}$-linear and multiplicative relations that they satisfy and on certain dimension considerations, do intersect the graph of $\exp$. This is interpreted rigorously through the technical conditions of ``freeness" and ``rotundity", which are expected to be sufficient for an algebraic subvariety of $\mathbb{C}^n \times (\mathbb{C}^\times)^n$ to intersect the graph of $\exp$. For example, the case $n=1$ states that if $p \in \mathbb{C}[X,Y]$ is a polynomial depending on both variables then there are infinitely many points $z \in \mathbb{C}$ such that $p(z, e^z)=0$; as discussed in \cite{Mar}, this can be proved by using standard techniques from complex analysis. Other cases of this conjecture have been addressed in several papers, see \cite{Zil02, Zil11, MZ, BM, DFT, K19}, each one imposing additional geometric restrictions on the varieties involved or considering modified versions of the problem.

We are particularly interested in \cite{Zil02} and \cite{Zil11}. There, Zilber proved a special case of the conjecture for $\exp$, studying intersections between the graph of $\exp$ and varieties of the form $L \times W$, satisfying the freeness and rotundity conditions mentioned above, where $L$ is an $\mathbb{R}$-linear subspace of $\mathbb{C}^n$ and $W$ is an algebraic subvariety of $(\mathbb{C}^{\times})^n$. In particular, \cite{Zil02} finds intersections for this kind of varieties assuming a uniform version of Schanuel's conjecture, while \cite{Zil11} removes the dependence on the conjecture but requires $L$ to be defined over some ``generic" subfield of the reals.

This case is usually referred to as ``raising to powers". The reason for this can be explained with a simple example: suppose $W \subseteq (\mathbb{C}^\times)^2$ is an algebraic curve, defined by an equation of the form $f(Y_1,Y_2)=0$, and $L$ is a linear subspace of $\mathbb{C}^2$, defined by an equation of the form $X_2=\alpha X_1$ for some complex number $\alpha$. Then, to find an intersection between the variety $L \times W$ and the graph of $\exp$, we have to solve the system:

$$\begin{cases}
z_2=\alpha z_1 \\
f(w_1,w_2)=0 \\
z_1=\exp(w_1)\\
z_2=\exp(w_2)\\
\end{cases}$$ which can easily be reduced to the single equation $$f(\exp(z),\exp(\alpha z))=0.$$ If we allow for a multivalued ``raising-to-power-$\alpha$" operator, which maps $w \in \mathbb{C}^\times$ to any value of $\exp(\alpha \log w)$, then this corresponds to finding all solutions of $$f(w,w^\alpha)=0$$ hence justifying the name. Similarly, all systems of this form (in any number of variables) can be rewritten as systems of exponential sums.

In recent years, the geometric nature of the exponential-algebraic closedness problem has led to the formulation of similar questions for other functions with similar properties. One fruitful line of research is the one concerning semiabelian varieties: Kirby in \cite{K09} and Bays and Kirby in \cite{BK18} have shown how to generalize the model theory of the complex exponential function to the exponentials of semiabelian varieties, and in particular formulated a conjecture that is analogous to Zilber's and would have similar model-theoretic consequences. The problem of finding intersections between algebraic varieties and the graph of the exponential of an abelian variety has therefore begun to attract attention.

At the same time, the model theory of other functions of interest in arithmetic geometry has been developed, such as the modular $j$ invariant, especially after a version of Ax's theorem for $j$ was proved in \cite{PT}. This has naturally led to the formulation of a conjecture on solutions of systems of equations which involve the modular $j$ function, and it seems likely, although at the moment there are no results available, that the question might also be asked for other uniformizers of Shimura varieties. Some cases involving the $j$ function have been addressed, often taking inspiration from previous work in the exponential case: for example, \cite{EH} contains an analogue for $j$ of the results contained in \cite{BM} and \cite{DFT}, and \cite{AK} of the results from \cite{K19}. Our results are new, and deal with cases which were not addressed before.

In this paper, drawing inspiration from Zilber's results in \cite{Zil02} and \cite{Zil11}, we solve some special cases of the existential closedness conjectures for exponential maps of abelian varieties and for the $j$ function. In particular, we show that the graphs of these functions intersect algebraic subvarieties which satisfy certain geometric constraints.

In Section 2 we consider abelian varieties. Using the fact that abelian varieties are complex Lie groups, and their subgroups are thus well-understood thanks to Cartan's closed subgroup theorem, we obtain intersections between the graph of the exponential $\exp$ of an abelian variety $A$ and a sufficiently generic product of the form $L \times W$ which satisfies the technical condition of \textit{rotundity}, which will be defined later. The genericity condition is stated in terms of the lattice $\Lambda$ such that $A \cong \mathbb{C}^g/\Lambda$.

\begin{thm212}
Let $A \cong \mathbb{C}^g/\Lambda$ be a complex abelian variety, $LA$ its Lie algebra, $\exp_LA \rightarrow A$ the exponential map, and let $L \times W$ be a rotund subvariety of $LA \times A$, where $L \leq LA$ is a linear subspace that is not contained in any hyperplane defined over the conjugate of the dual of the lattice $\Lambda$. 

Then $\exp(L) \cap W$ is dense in $W$ in the Euclidean topology.
\end{thm212}

The additional assumption on $L$ makes $\exp(L)$ a dense subgroup of $A$ in the Euclidean topology. The exponential-algebraic closedness conjecture for abelian varieties predicts that algebraic subvarieties of the tangent bundle of $A$ which staisfy rotundity and another condition known as \textit{freeness} intersect the graph of the exponential of $A$. Our assumption in $L$ is stronger than freeness; however, it allows us to draw the stronger conclusion that $\exp(L) \cap W$ is Euclidean dense in $W$ and not just Zariski-dense.

In Section 3 we obtain a similar result for $j$; the assumption on subvarieties of the domain is now that they are \textit{M\"obius subvarieties} of $\mathbb{H}^n$ ($\mathbb{H}$ denotes the complex upper half plane.) These are intersections with $\mathbb{H}^n$  of complex algebraic varieties which are defined by using the action of the group $\sl_2(\mathbb{R})$ on $\mathbb{H}$ by M\"obius transformations: a matrix $g=\mat{a}{b}{c}{d}$ acts on $z$ by taking it to $gz=\frac{az+b}{cz+d}$, and a M\"obius variety is a variety that can be defined by only using conditions of the form $z_i=c$ (constant coordinates) and $z_j=gz_i$. The geometry of M\"obius transformations, and other general facts about the $j$ function, will be reviewed at the beginning of Section 3. The statement relies on the technical conditions of \textit{freeness} and \textit{broadness}.

\begin{thm328}
Let $L \times W$ be a free broad algebraic subvariety of $\mathbb {H}^n  \times \mathbb{C}^n$ such that $L \subseteq \mathbb{H}^n$ is a M\"obius subvariety. Then $W$ contains a dense subset of points of $j(L)$.
\end{thm328}

The link between Theorem \ref{Main} and Zilber's raising to powers should be explained. Consider a M\"obius subvariety $L$ of $\mathbb{H}^2$, defined by the equation $X_2=gX_1$ for a matrix $g \in \sl_2(\mathbb{R})$, and an algebraic subvariety $W \subseteq \mathbb{C}^2$, defined by a polynomial $f(Y_1,Y_2)$. As in the exponential case, it is easy to see how intersecting the variety $L \times W$ with the graph of $j$ is equivalent to solving the single equation $f(j(z),j(gz))=0$. Therefore, if we define for $g$ a multivalued operator $w \mapsto w^g$ on $\mathbb{C}$ which maps $w$ to any value of $gj^{-1}(z)$ (analogously to what we did for $\exp$), we can see the equation above as $f(w,w^g)=0$. In this sense, this kind of systems of equations represent an analogue of the exponential raising to powers for the $j$ function.

As the $j$ function is algebraically independent from its first two derivatives, results in this area often take into consideration equations which involve $j, j'$ and $j''$ (see for example \cite{AK}). The methods of this paper do not seem to work if we consider all three functions together, and only allow us to work with one function at a time, yielding only some partial results. This is explained in Section 4, where we sketch the proof of a partial result that can be obtained for $j'$.

While the motivation for these questions comes from model theory, the methods that we use are mainly geometric, ergodic- and number-theoretic. Hence, we have tried to make this exposition as self-contained as possible.

\textbf{Acknowledgements.} The author would like to thank his supervisor, Vincenzo Mantova, for suggesting to work on this topic, for his countless suggestions, and for the valuable feedback that he has provided. Thanks to Gareth Jones, whose suggestion to use Ratner's theorem to tackle Proposition \ref{dense} proved correct, and to Vahagn Aslanyan and Jonathan Kirby for many interesting conversations around these topics. Many thanks to the referee for their many suggestions on the manuscripit and how to improve it. This research was done as part of the author's Ph.D. project, supported by a scholarship from the School of Mathematics of the University of Leeds.

\section{Abelian Varieties}

\subsection{Background and Notation}

In this section we will deal with complex abelian varieties. Throughout, $A$ will be used to denote an abelian variety, which will be identified with the set $A(\mathbb{C})$ of its complex points. 

The Lie algebra of $A$ (its tangent space at identity) will be denoted $LA$, and the tangent bundle will be identified with $LA \times A$ and denoted $TA$. These will also be identified with the sets $LA(\mathbb{C})$ and $TA(\mathbb{C})$ of their complex points. Given an abelian subvariety $B$ of $A$, we will denote by $LB$ its Lie algebra (identified with a subspace of $LA$) and by $TB$ its tangent bundle, identified with $LB \times B$ as a subset of $TA$. We note that the quotient $A/B$ is also an abelian variety, with Lie algebra isomorphic to $LA/LB$ and tangent bundle isomorphic to $TA/TB$.

There is an analytic covering map $\exp: LA\rightarrow A$, known as the \textit{exponential map} of $A$. $LA$ is a complex vector space of the same dimension $g$ as the abelian variety, and as such it is isomorphic to $\mathbb{C}^g$; the exponential map is used to identify $A$ with the quotient $\mathbb{C}^g/\Lambda$, where $\Lambda$ is a lattice in $\mathbb{C}^g$. 

We will use $+$ and $-$ to denote the operations on the Lie algebra, while $\oplus$ and $\ominus$ will denote the operations on both $A$ and the tangent bundle $TA$ (it will be clear from context which one we will be referring to.)

Unless explicitly stated otherwise, all algebraic varieties considered in this section will be irreducible.

\subsection{Dense Subgroups of Abelian Varieties}

In this section we describe the linear subspaces of $LA$ whose image under the exponential function is dense in $A$ in the Euclidean topology, showing that this property holds for a generic subspace. This will allow us to apply a similar argument to the one used in \autocite[Section 6]{K19}. This characterization is not new as it appears implicitly for example in \autocite[Section 3]{UY18}, but we work out an explicit statement.

We endow $\mathbb{C}^g$ with the usual Hermitian product, defined for $z=(z_1,\dots, z_g)$ and $w=(w_1,\dots,w_g)$ as $z \cdot w=z_1\o{w}_1 + \dots + z_g\o{w}_g$. Consider, for complex vectors $z,z'$, $\langle z,z' \rangle:=\Re(z \cdot z')$.

\begin{rem}\label{scalar}
	Identifying $z=(z_1,\dots,z_g)=(x_1+iy_1,\dots,x_g+iy_g)$ with the real vector $(x_1,y_1,\dots,x_g,y_g) \in \mathbb{R}^{2g}$, we have that $\langle z, z' \rangle$ coincides with the real scalar product in $\mathbb{R}^{2g}$, as $$\Re((x_1+iy_1,\dots,x_g+iy_g) \cdot (x_1'+iy_1',\dots,x_g'+iy_g'))=$$ $$=x_1x_1'+y_1y_1'+\dots +x_gx_g'+y_gy_g'.$$
\end{rem} 

\begin{defn}
	Let $\Lambda$ be a lattice of rank $2g$ in $\mathbb{C}^g$. The \textit{dual lattice} $\Lambda^*$ is the lattice defined by $$\Lambda^*:=\{\theta \in \mathbb{C}^g|\langle \theta, \lambda \rangle \in \mathbb{Z} \, \, \forall \lambda \in \Lambda \}.$$
\end{defn}

We use the dual lattice to describe the real hyperplanes of $\mathbb{C}^g$ whose image in $\mathbb{C}^g/\Lambda$ is closed.

\begin{lem}[See {\autocite[Section 3]{UY18}}]
	Let $H$ be a real hyperplane in $\mathbb{C}^g$. Then $H + \Lambda$ is closed in $\mathbb{C}^g/\Lambda$ if and only if $H$ can be defined by an equation of the form $\Re\left(\sum_{i=1}^{g} \o{\theta}_i z_i \right)=0$ for $\theta=(\theta_1,\dots,\theta_g) \in \Lambda^*$.
\end{lem}

\begin{proof}
	As a real torus, $\mathbb{C}^g/\Lambda$ is isomorphic to $\mathbb{R}^{2g}/\mathbb{Z}^{2g}$. It is well-known that the hyperplanes $H$ of $\mathbb{R}^{2g}$ which have closed image in the quotient are those defined by a $\mathbb{Q}$-linear equation, as these have a set of generators in $\mathbb{Z}^{2g}$; in other words, $H \cap \mathbb{Z}^{2g}$ is a lattice in $H$. Hence, in the setting of $\mathbb{C}^g /\Lambda$ we must characterize the spaces $H$ for which $H \cap \Lambda$ is a lattice in $H$.
	
	So suppose $H \cap \Lambda$ is a lattice in $H$. Then $H$ contains $2g-1$ linearly independent elements of $\Lambda$, and there is an element of $\Lambda^*$ which is orthogonal to each of those for the real scalar product. Then by Remark \ref{scalar} we have the claim.
	
	Conversely, suppose $H$ is defined by an equation of the form above. Then $H$ is orthogonal to a space defined by an element of the dual, and so it is generated by points in the lattice.
\end{proof}

This property easily translates to complex hyperplanes.

\begin{coro}
	Suppose $L$ is a complex hyperplane of $\mathbb{C}^g$ which cannot be defined by an equation of the form $\sum_{i=1}^{g} \o{\theta}_i z_i =0$ for $\theta=(\theta_1,\dots,\theta_g) \in \Lambda^*$ (we say such an hyperplane is not defined over the conjugate of the dual lattice). Then $L+ \Lambda$ is dense in $\mathbb{C}^g/\Lambda$.
\end{coro}

\begin{proof}
	If a complex linear space satisfies $\Re\left(\sum_{i=1}^{g} \o{\theta}_i z_i \right)=0$ for coefficients in $\Lambda^*$, then it must satisfy $\Im\left(\sum_{i=1}^{g} \o{\theta}_i z_i \right)=0$, and hence the equation $\sum_{i=1}^{g} \o{\theta}_i z_i =0$. 
\end{proof}

The next lemma easily follows.

\begin{lem}\label{hyperdensecomplex}
Let $L$ be a complex linear subspace of $\mathbb{C}^g$. Then $L + \Lambda$ is dense in $\mathbb{C}^g/\Lambda$ if and only if $L$ is not contained in a complex hyperplane defined over the conjugate of $\Lambda^*$. 
\end{lem}

As an example, consider the case of a power of an elliptic curve $E \cong \mathbb{C}/\Lambda$, where $\Lambda$ is a lattice of the form $\mathbb{Z}+\tau \mathbb{Z}$ for some complex number $\tau=a+ib$.

Identifying $\mathbb{C}$ with $\mathbb{R}^2$, the lattice $\Lambda$ is generated by the points $\vect{1}{0}$ and $\vect{a}{b}$. Clearly then a pair of generators of $\Lambda^*$ is given by $\vect{0}{\frac{1}{b}}$ and $\vect{1}{-\frac{a}{b}}$. These correspond to the complex numbers $-\frac{i}{b}$ and $1+\frac{ia}{b}$. Note that, multiplying both for $ib$, we obtain 1 and $-a+ib=-\o{\tau}$, which generate the conjugate of $\Lambda$. Hence, in the case of powers of an elliptic curve, the conjugate of the dual lattice is actually the original lattice up to multiplication by a scalar, and Lemma \ref{hyperdensecomplex} takes the following form.

\begin{coro}\label{elliptichyperdense}
	Let $E \cong \mathbb{C}/\Lambda$ be an elliptic curve, $L \leq \mathbb{C}^g$ a complex linear space, $\exp$ the exponential map of the abelian variety $E^g$. Then $\exp(L)$ is dense in $E^g$ if and only if $L$ is not contained in an hyperplane defined over $\Lambda$.
\end{coro}

Recall that an elliptic curve $E$ is said to have \textit{Complex Multiplication} when its endomorphism ring is strictly larger than $\mathbb{Z}$, and that these are characterized as the curves whose period lattice has the form $\mathbb{Z}+\tau \mathbb{Z}$ for $\tau$ an imaginary quadratic number (see \cite{HS}, Example A.5.1.3). If $E$ has CM, linear spaces defined over the lattice are Lie algebras of abelian subvarieties of $E^g$, and the corollary can be seen as stating that all free complex hyperplanes have dense exponentials.

\begin{rem}
A similar description can be given by the well-known characterization of complex abelian varieties as complex tori which admit a \textit{polarization}, that is, an Hermitian form on $\mathbb{C}^g$ whose imaginary part takes integer values on $\Lambda \times \Lambda$. We chose to follow \cite{UY18} in using the dual lattice as it provides a more concrete description.
\end{rem}

\subsection{Intersections}

In this section $L$ will always be a translate of a linear subspace of $LA$.

\begin{defn}
	Let $L \times W \subseteq TA$ be an algebraic subvariety with $L \leq LA$ a linear subspace and $W \subseteq A$ an algebraic subvariety.
 
    We say $L \times W$ is \textit{rotund} if for every abelian subvariety $B$ and quotient map $\pi_{TB}:TA \twoheadrightarrow TA/TB$, $$\dim \pi_{TB}(L \times W) \geq \dim A/B.$$
\end{defn}

Note that in particular this implies that $\dim L + \dim W \geq \dim A$, by considering the trivial subgroup of $A$. If $A$ is simple, that is, it has no proper non-trivial abelian subvariety, this is the only condition.

The proof of the next two statements are similar to the proofs of \autocite[Lemma 6.1 and Proposition 6.2]{K19}, translated to the context of abelian varieties.

\begin{prop}\label{fibdim}
	Let $L \times W \subseteq TA$ be a rotund subvariety with $L \leq LA$ a linear subspace and $W \subseteq A$ an algebraic subvariety, such that $\dim L + \dim W = \dim A$, and $B$ an abelian subvariety of $A$. Then there is a non-empty Zariski-open subset $V_B$ of $L \times W$ such that if $\gamma \in V_B$, $$\dim((L \times W) \cap (\gamma \oplus TB)) \leq \dim B.$$
\end{prop}

\begin{proof}
	Let $\pi_{TB}:TA \twoheadrightarrow TA/TB$ be the quotient map. Consider the restriction of $\pi_{TB}$ to $L \times W$, so that for every $\gamma \in L \times W$, $((\gamma \oplus TB) \cap (L \times W))$ is a fibre of $\pi_{TB}$. Then the fibre dimension theorem implies that there is a Zariski-open subset $V_B$ of $L \times W$, such that for all $\gamma \in V_B$, 
    \begin{align*}
        \dim ((L \times W) \cap (\gamma \oplus TB))&=\dim \pi_{TB}^{-1}(\gamma) \\
                                                &=\dim L + \dim W - \dim (\pi_{TB}(L \times W))\\
                                                &\leq \dim L + \dim W - \dim A/B\\
                                                &=\dim A-\dim A+\dim B\\
                                                &=\dim B
    \end{align*}
where the inequality follows from rotundity.
\end{proof}	

\begin{thm}\label{interab}
Let $L \times W \subseteq LA \times A$ be a rotund subvariety with $L \leq LA$ a linear subspace and $W \subseteq A$ an algebraic subvariety, and suppose $L$ is not contained in any hyperplane defined over the conjugate of the dual lattice. Then $\exp(L) \cap W$ is dense in $W$ in the Euclidean topology.
\end{thm}

\begin{proof}
	By intersecting $L$ with generic hyperplanes if necessary, we assume that $\dim L + \dim W =\dim A$.
	
	Consider the function $\theta: L \times W \rightarrow A$, mapping $(l, w)$ to $w \ominus \exp(l)$. Let $a=(l_0, w_0)$ be a point in $L \times W$. Consider a small open neighbourhood $U$ of $a$ in $L \times W$ in the Euclidean topology, and the restriction of $\theta$ to $U$, denoted $\theta_{|U}$. 
	
	The set $\theta_{|U}^{-1}(\theta(a))$ is analytic, and thus we can assume it is connected by shrinking $U$ if necessary; if it has dimension 0, then $\theta_{|U}$ is locally finite-to-one and therefore open, by the Remmert open mapping theorem. So, under the assumption that $\dim\theta_{|U}^{-1}(\theta(a))=0$, $\theta(U)$ is an open subset of $A$, and as $\exp(L)$ is dense in $A$ by Lemma \ref{hyperdensecomplex}, $\theta(U) \cap \exp(L) \neq \varnothing$. But then there is a point $(l, w) \in U$ such that $w \ominus \exp(l)=\exp(l') \in \exp(L)$, and so $w =\exp(l +l') \in \exp(L)$.
	
	Therefore it remains to show that $\theta_{|U}^{-1}(\theta(a))$ can be taken to have dimension 0 for almost every $a \in L \times W$. Suppose then that for some $a$, the set $$S:=\{(l, w) \in L \times W \mid \theta(l, w)= \theta(a) \}$$ has positive dimension.
	
	For any $t \in LA$ such that $\exp(t)=\theta(a)$, consider the set $$S':=\{(l + t, w) \in LA \times A \mid (l, w) \in S \}.$$ For any point $b \in S'$, we have $$\theta(b)=w \ominus (\exp(l) \oplus \exp(t))=\theta(a)\ominus \theta(a)=0_A$$ so $S'$ is a positive dimensional component of the intersection $$((t + L) \times W) \cap \Gamma_\exp$$ where $\Gamma_\exp$ denotes the graph of $\exp$. 
	
	Recall that such a component is said to be \textit{atypical} when $\dim S > \dim L + \dim W -\dim A$. In this case, since $\dim L+\dim W=\dim A$ by assumption, every positive dimensional component is atypical.
	
	Consider the family $\mathcal{V}$ of subvarieties of $TA$ defined by $\mathcal{V}:=\{(s+L) \times W | s \in LA \}$. By \autocite[Theorem 8.1]{K06}, there is a finite set $\mathcal{B}$ of proper abelian subvarieties of $A$ such that for every subvariety $(s+L) \times W \in \mathcal{V}$, every positive dimensional component of the intersection $((s + L) \times W) \cap \Gamma_\exp$ is contained in a translate of $TB$ for some $B \in \mathcal{B}$. Hence there are $B \in \mathcal{B}$ and $\gamma \in TA$ such that $S' \subseteq \gamma \oplus TB$, and therefore $$S \subseteq (\gamma \ominus (t, 0_A)) \oplus TB.$$ Since $a \in S$, this can be rewritten as $S \subseteq a \oplus TB.$ Such a $B$ can be chosen to be minimal, and by the contrapositive of the same theorem we must have that $$\dim S= \dim ((a \oplus TB) \cap (L \times W)) - \dim B.$$
	
	Since $\dim S$ is positive, $a \in V_B$ (in the notation of Proposition \ref{fibdim}). Then, to avoid that $\dim S >0$, we just need to take $a \in \bigcap_{B \in \mathcal{B}}V_B$.
\end{proof}

\begin{rem}
    The exponential algebraic closedness conjecture for abelian varieties predicts that algebraic subvarieties of $TA$ which satisfy rotundity and another condition, known as \textit{freeness}, have a Zariski-dense subset of points of the form $(z,\exp(z))$. For varieties of the form $L \times W$, freeness requires that $L$ is not contained in a translate of the Lie algebra of an abelian subvariety of $A$ and that $W$ is not contained in a translate of an abelian subvariety of $A$; hence, our genericity condition on $L$ is stronger than the first part of the definition of freeness. However, we can also draw from it a stronger conclusion, namely Euclidean density of $\exp(L) \cap W$ in $W$. 
\end{rem}

Using Corollary \ref{elliptichyperdense}, we note that in the case in which $A$ is a power of an elliptic curve Theorem \ref{interab} the notion of freeness is actually implied by our assumption, and so the corresponding result of exponential algebraic closedness-type follows.

\begin{coro}
Let $E$ be an elliptic curve. Let $L \times W$ be a free and rotund subvariety of $LE^g \times E^g$, with $L \leq LE^g$ a linear subspace not defined over $\End(E) \otimes \mathbb{R}$ and $W \subseteq E^g$ an algebraic subvariety. 

Then $\exp(L) \cap W$ is dense in $W$ in the Euclidean topology.
\end{coro}

\section{The $j$ Function}\label{j}

\subsection{Background and Notation}

Let $\mathbb{H}$ denote the complex upper half plane, i.e$.$ the set $\{z \in \mathbb{C}\mid \Im(z) >0 \}$. The special linear group $\sl_2(\mathbb{R})$ acts on $\mathbb{H}$ by M\"obius transformations, that is, the matrix $g=\mat{a}{b}{c}{d}$ acts on $z \in \mathbb{H}$ by mapping it to $\frac{az+b}{cz+d}$.

For the rest of this subsection, let $G=\sl_2(\mathbb{R})$ and $\Gamma=\sl_2(\mathbb{Z})$.

\begin{rem}\label{action}
\begin{itemize}
    \item[(a)] The action of $G$ on $\mathbb{H}$ easily extends to an action on $\mathbb{H} \cup \mathbb{R} \cup \{\infty\}$. 
    \item[(b)] We will abuse notation and say that a matrix $g \in G$ is in $\gl_2^+(\mathbb{Q})$ if it is a scalar multiple of such a matrix, that is, if there is $\lambda \in \mathbb{R}_{>0}$ such that $\lambda g \in \gl_2(\mathbb{Q})^+$.
\end{itemize}
\end{rem}

Recall that an element of $G$ is called \textit{parabolic} if it is conjugate to a matrix of the form $\mat{1}{t}{0}{1}$ with $t \neq 0$, and that a \textit{cusp} for a discrete subgroup $\Gamma'$ of $G$ is a point $x \in \mathbb{R} \cup \{\infty\}$ such that $\gamma x=x$ for some parabolic $\gamma \in \Gamma'$. It is not hard to check that the set of cusps for $\Gamma$ is $\mathbb{Q} \cup \{ \infty\}$.

\begin{defn}
Let $f$ be a $\Gamma$-invariant function on $\mathbb{H}$, $\mathbb{D}$ denote the open unit disk. Then let $f^*:\mathbb{D} \setminus \{0\} \rightarrow \mathbb{C}$ be the function $f^*(\exp(2  \pi i z))=f(z)$ (this is well-defined because if $\exp(2\pi iz_1)=\exp(2\pi i z_2))$ then $z_1-z_2 \in \mathbb{Z}$, so $f(z_1)=f(z_2)$ by $\Gamma$-invariance.) 

We say $f$ is \textit{meromorphic} (resp$.$ \textit{has a pole of order $n$}) \textit{at the cusp} if $f^*$ is meromorphic (resp$.$ has a pole of order $n$) at 0.
\end{defn}

We recall the definition of the $j$ function, following \cite{Mil}, Part I, Section 4.

\begin{defn}
The modular invariant $j$ is the unique function $j:\mathbb{H} \rightarrow \mathbb{C}$ that is holomorphic, $\Gamma$-invariant, with a simple pole at the cusp and such that $j(i)=1728$ and $j(e^{\frac{i\pi}{3}})=0$. 
\end{defn}

As in the previous section, unless explicitly stated otherwise, all algebraic varieties will be irreducible.

By \textit{algebraic subvariety of $\mathbb{H}^n$} we mean the intersection with $\mathbb{H}^n$ of an algebraic subvariety of $\mathbb{C}^n$. Similarly, by an \textit{algebraic subvariety of $\mathbb{H}^n \times \mathbb{C}^n$} we mean the intersection with $\mathbb{H}^n \times \mathbb{C}^n$ of an algebraic subvariety of $\mathbb{C}^{2n}$.

\begin{defn}
Let $L$ be an algebraic subvariety of $\mathbb{H}^n$. We say $L$ is a \textit{M\"obius subvariety of $\mathbb{H}^n$} if it can be defined by conditions of the form $z_k=gz_i$ for $g \in G, i,k \leq n$ and $z_i=c$ for $c \in \mathbb{H},i \leq n $.
\end{defn}

For example the variety $\{(z_1,z_2) \in \mathbb{H}^2 |z_2=z_1+\sqrt{2} \}$ is a M\"obius subvariety of $\mathbb{H}^n$, defined by $z_2=gz_1$ where $g=\mat{1}{\sqrt{2}}{0}{1}$.

\begin{defn}
A M\"obius subvariety $L$ of $\mathbb{H}^n$ is \textit{weakly special} if it can be defined by using just conditions specifying constant coordinates and equations of the form $z_i=hz_k$ where $h \in \gl_2(\mathbb{Q})^+$.
\end{defn}

Weakly special varieties are M\"obius varieties of particular interest, as they are an example of bialgebraic varieties in the sense of \cite{KUY}.

\begin{fact}[\cite{UY}, Theorem 1.2]
Let $V$ be an algebraic subvariety of $\mathbb{H}^n$. Then $j(V)$ is an algebraic subvariety of $\mathbb{C}^n$ if and only if $V$ is weakly special.
\end{fact}

We call the image of a weakly special subvariety of $\mathbb{H}^n$ under $j$ a \textit{weakly special subvariety of $\mathbb{C}^n$}.

The corresponding algebraic relations in the codomain are given by the \textit{modular polynomials}.

\begin{defn}\label{modpol}
For each natural number $N$, the $N$-th \textit{modular polynomial} is the polynomial $\Phi_N \in \mathbb{Z}[X,Y]$ such that for all $w_1,w_2 \in \mathbb{C}$, $\Phi_N(w_1,w_2)=0$ holds if and only if there are $z \in \mathbb{H}$ and $g \in \gl_2(\mathbb{Q})^+$ with coprime integer entries and determinant $N$ such that $w_1=j(z)$, $w_2=j(gz)$. 
\end{defn}

\subsection{Density of Images of M\"obius Subvarieties}\label{rat}

\begin{defn}
A M\"obius subvariety $L$ of $\mathbb{H}^n$ is \textit{free} if it is not contained in any weakly special subvariety, i.e$.$, if no coordinate is constant on $L$ and for no $i,k \leq n$ there is $h \in \gl_2(\mathbb{Q})^+$ such that $z_k=hz_i$ for all $(z_1,\dots,z_n) \in L$.
\end{defn}

\begin{prop}\label{dense}
Let $L \subseteq \mathbb{H}^n$ be a free M\"obius subvariety. Then $j(L)$ is dense in the Euclidean topology in $\mathbb{C}^n$.
\end{prop}

Proposition \ref{dense} can be regarded as a standard application of Ratner's theorem on unipotent flows (see \cite{Rat, Mor}). 

We will actually need a different statement, which gives density not just of $j(L)$ but of the $\sl_2(\mathbb{Z})^n$-orbit of a subgroup of $\sl_2(\mathbb{R})^n$ obtained from $L$ and is a consequence of work of Ullmo and Yafaev on applications of Ratner's theorem to flows on Shimura varieties (see \cite{UY18b}). 

\begin{lem}\label{densityinsl2}
    Given $g_2,\dots,g_n \in \sl_2(\mathbb{R})$, let $F=\{(h,g_2hg_2^{-1},\dots,g_nhg_n^{-1}) \in \sl_2(\mathbb{R})^n \mid h \in \sl_2(\mathbb{R})\}$. 
    
    If the M\"obius subvariety $L=\{(z,g_2z,\dots,g_nz) \in \mathbb{H}^n \mid z \in \mathbb{H}^n\}$ of $\mathbb{H}^n$ is free, then $\sl_2(\mathbb{Z})^n \cdot F$ is dense in $\sl_2(\mathbb{R})^n$. 
\end{lem}

\begin{proof}
    By \autocite[Proposition 5.2]{UY18b}. Note that freeness of $L$ implies that the Mumford-Tate group of $F^+(\mathbb{R})$ is $\sl_2(\mathbb{R})$.
\end{proof}

\begin{proof}[Proof of Proposition \ref{dense}]
    It is sufficient to prove this for $L \subseteq \mathbb{H}^n$ a one-dimensional M\"obius subvariety, so assume $L=\{(z,g_2z,\dots,g_nz) \in \mathbb{H}^n \mid z \in \mathbb{H}\}$ for some $g_2\dots,g_n$.
    
    Let $ \varnothing \neq U \subseteq \mathbb{C}^g$ be an open set. Let $w=(w_1,\dots,w_n) \in U$. Since $j$ is surjective, there is $z=(z_1,\dots,z_n) \in \mathbb{H}^n$ such that $j(z)=w$. Since the action of $\sl_2(\mathbb{R})$ on $\mathbb{H}$ is transitive, there exist $h_2,\dots,h_n \in \sl_2(\mathbb{R})$ such that $z_i=h_ig_iz_1$ for $i=2,\dots,n$. By continuity of the action, there is an open subset $U' \subseteq \sl_2(\mathbb{R})^{n}$ containing $(\mathbb{I}_2,h_2,\dots,h_n)$ such that if $(h_1',h_2',\dots,h_n') \in U'$, then $$(j(h_1'z),j(h_2'g_2z), \dots,j(h_n'g_nz)) \in U.$$

    By Lemma \ref{densityinsl2}, there are $h\in \sl_2(\mathbb{R})$ and $(\gamma_1,\dots,\gamma_n) \in \sl_2(\mathbb{Z})^n$ such that $$(\gamma_1h, \gamma_2g_2hg_2^{-1}, \dots, \gamma_ng_nhg_n^{-1}) \in U'.$$ Hence, by $\sl_2(\mathbb{Z})$-invariance of $j$, $$(j(hz), j(g_2hz), \dots, j(g_nhz))=(j(\gamma_1hz), j(\gamma_2g_2hz), \dots, j(\gamma_ng_nhz)) \in U.$$
\end{proof}

\subsection{Intersections}

In this subsection we prove the existence of intersections between images of free M\"obius subvarieties and appropriate algebraic subvarieties of $\mathbb{C}^n$. 

\begin{lem}\label{sincos}
Let $z_1, z_2 \in \mathbb{H}$, with $z_l=x_l+iy_l$ for $l=1,2$. Then for any $c,d \in \mathbb{R}$ such that $|cz_1+d|^2=\frac{y_1}{y_2}$, there are $a,b \in \mathbb{R}$ such that $\mat{a}{b}{c}{d}z_1=z_2$.
\end{lem}

\begin{proof}
Direct calculations show that $\mat{\sqrt{y}}{\frac{x}{\sqrt{y}}}{0}{\frac{1}{\sqrt{y}}} i= x+iy$, for all $x$ and $y$, and that $gi=i$ for all $g \in \so_2(\mathbb{R})$.

Therefore, for all $\theta \in [0,2\pi)$ we have that $$\mat{\sqrt{y_2}}{\frac{x_2}{\sqrt{y_2}}}{0}{\frac{1}{\sqrt{y_2}}}  \mat{\cos \theta}{\sin \theta}{-\sin \theta}{\cos \theta} \mat{\frac{1}{\sqrt{y_1}}}{-\frac{x_1}{\sqrt{y_1}}}{0}{\sqrt{y_1}}z_1=z_2$$ and that the lower entries of the product matrix are $-\frac{\cos\theta}{\sqrt{y_1y_2}}$  and $x_1 \frac{\cos \theta}{\sqrt{y_1y_2}}+(\sin \theta) \sqrt{\frac{y_1}{y_2}}$. Therefore, $$cz_1+d=\sqrt{\frac{y_1}{y_2}}(\sin\theta-i\cos\theta).$$
\end{proof}

This together with density of images of M\"obius subvarieties is enough to prove the existence of intersections in the case $\dim L=\codim W=1$.

\begin{prop}\label{d=1}
	Let $L \times W$ be an algbraic subvariety of $\mathbb{H}^n \times \mathbb{C}^n$ such that $L$ is a free M\"obius subvariety of $\mathbb{H}^n$ of dimension $1$, $W \subseteq \mathbb{C}^n$ is algebraic and none of its coordinates is identically $0$ or $1728$, and $\dim L + \dim W \geq n$. Then $W$ has a dense subset of points of $j(L)$. 
\end{prop}

We will obtain this as a corollary of a stronger result, which does not require $W$ to be an algebraic variety.

\begin{lem}\label{holo}
	Let $g_2,\dots,g_n \in \sl_2(\mathbb{R})$; $V$ an open subset of $\mathbb{C}^n$, $f:V \rightarrow \mathbb{C}$ a holomorphic function. Denote by $W$ the zero locus of $f$; assume that $W$ has a regular point $(w_1,\dots,w_n)$ such that $w_i \notin \{0,1728\}$ for all $i=1,\dots,n$. Then $W$ has a dense subset of points of the form $(j(z_1), j(g_2z_1), \dots, j(g_nz_1))$.
\end{lem}

\begin{proof}
	Let $L=\{(z_1,\dots, z_n) \mid z_i=g_iz_1 \textnormal{ for } i=2,\dots, n \}$. 
	
	Let $(w_1,\dots,w_n) \in W$ be a regular point such that no coordinate of $w$ is 0 or 1728. Then find a point $(z_1,\dots,z_n) \in \mathbb{H}^n$ such that $j(z_i)=w_i$ for $i=1,\dots,n$, and $h_2,\dots, h_n \in \sl_2(\mathbb{R})$ such that $z_i=h_ig_iz_1$ for $i=2,\dots,n$. Then consider the function $G:\mathbb{H} \rightarrow \mathbb{C}^n$, mapping $z \in \mathbb{H}^n$ to $f(j(z),j(h_2g_2z),\dots,j(h_ng_nz))$. By construction $G(z_1)=0$. Let $U \subseteq \mathbb{H}$ be a neighbourhood of $z_1$.
	
	By Lemma \ref{densityinsl2}, there is a sequence $\{\o{g}^i \}_{i \in \mathbb{N}}$, converging to $(\mathbb{I}_2,h_2,\dots,h_n)$, such that each $\o{g}^i$ is a tuple of the form $$(\gamma_1^{i}k_i,\gamma_2^ig_2k_ig_2^{-1},\dots,\gamma_n^ig_nk_ig_n^{-1})$$ for some $k_i \in \sl_2(\mathbb{R})$. 
	
	Then consider the sequence of functions $\{G_i\}_{i \in \mathbb{N}}$, where each $G_i:U \rightarrow \mathbb{C}$ is defined by $G_i(z)=f(j(\gamma_1^ik_iz),j(\gamma_2^ig_2k_iz),\dots,j(\gamma_n^ig_nk_iz))$. It is then clear that $$\lim_{i \in \mathbb{N}} G_i(z_1)=G(z_1)=0.$$
	
	If $\frac{d}{dz} G(z_1) \neq 0$, $z_1$ is an isolated zero of $G$, and therefore we may assume (restricting $U$ if necessary) that it is the only zero of $G$ in $\o{U}$, and then by the argument principle $\int_{\partial U} \frac{G'(z)}{G(z)}dz=2\pi i$. Then for sufficiently large $i$ we have $\int_{\partial U} \frac{G_i'(z)}{G_i(z)}dz=2\pi i$, and therefore $G_i$ has a zero $z_1' \in U$: hence $$f(j(\gamma_1^i k_iz_1'),j(\gamma_2^ig_2k_iz_1'),\dots, j(\gamma_n^ig_nk_iz_1'))=0,$$ that is, $$(j(k_iz_1'),j(g_2k_iz_1'),\dots,j(g_nk_iz_1')) \in j(L) \cap W.$$

    Hence it remains to prove the following claim.
 
	\textbf{Claim}: Without loss of generality we may assume $\frac{d}{dz} G(z_1) \neq 0$.
	
	\textbf{Proof of Claim}: Suppose $\frac{d}{dz} G(z_1) = 0$. Writing $j(\o{hg}z_1)$ for the tuple $$(j(z_1),j(h_2g_2z_1),\dots,j(h_ng_nz_1)),$$ we have that $$\frac{d}{dz}G(z_1)=$$ $$=\frac{\partial f}{\partial Y_1}(j(\o{hg}z_1))j'(z_1)+ \frac{\partial f}{\partial Y_2}(j(\o{hg}z_1))(j(h_2z_1))'+ \dots + \frac{\partial f}{\partial Y_n}(j(\o{hg}z_1))(j(h_nz_1))'$$ where if $h_ig_i=\left( \begin{matrix} a & b \\ c & d \end{matrix} \right)$, $(j(h_ig_iz_1))'=\frac{j'(h_ig_iz_1)}{(cz_1+d)^2}$.
 
     Since $(w_1,\dots,w_n)=j(\o{hg}z_1)$ is a regular point in $W$ and none of the $w_i$'s is 0 or 1728, the summands in $\frac{d}{dz} F(z_1)$ are not all zero. Then in particular at least two of them are not zero, and so there is $l>1$ such that $\frac{\partial f}{\partial Y_l}(j(\o{hg}z_l))\frac{j'(h_lg_lz_1)}{(cz_1+d)^2} \neq 0$. Then it is enough to change the matrix $h_l$, which we are free to do by Lemma \ref{sincos}: there is a matrix $k$ such that $kg_lz_1=h_lg_lz_1$, but $(j(kg_lz_1))' \neq (j(h_lg_lz_1))'$.
\end{proof}

\begin{proof}[Proof of Proposition \ref{d=1}]
	If $\dim W=n$ then it follows directly from Proposition \ref{dense}. If $\dim W=n-1$, apply Lemma \ref{holo} in the case where $V=\mathbb{C}^n$ and $f$ is a polynomial.
\end{proof}

To extend this to higher dimension we need the notions of broadness and freeness in the codomain, introduced by Aslanyan in \cite{A}. 

\begin{defn}\label{broad}
	Let $L \times W$ be an algebraic subvariety of $\mathbb{H}^n \times \mathbb{C}^n$ such that $L$ is a M\"obius subvariety of $\mathbb{H}^n$ and $W$ is an algebraic subvariety of $\mathbb{C}^n$.
	
	For every $I=(i_1,\dots,i_k)$ where $1 \leq i_1 < \dots < i_k \leq n$ are natural numbers, let $\pi_I$ denote the projection $$\pi_I:\mathbb{H}^n \times \mathbb{C}^n \rightarrow \mathbb{H}^k \times \mathbb{C}^k$$ which maps $(z_1,\dots,z_n,w_1,\dots,w_n)$ to $(z_{i_1}, \dots z_{i_k}, w_{i_1} \dots, w_{i_k})$.
	
	We say $L \times W$ is \textit{broad} if for every $I$, $$\dim \pi_I(L \times W) \geq |I|.$$
\end{defn}

\begin{defn}\label{free}
The algebraic variety $W \subseteq \mathbb{C}^n$ is \textit{free} if it is not contained in a weakly special subvariety of $\mathbb{C}^n$.

The algebraic subvariety $L \times W$ of $\mathbb{H}^n \times \mathbb{C}^n$ is \textit{free} if both $L$ and $W$ are free. If a subvariety of $\mathbb{H}^n \times \mathbb{C}^n$ is both broad and free, we will say it is \textit{free broad.}
\end{defn}

The \textit{Existential Closedness Conjecture for $j$} (see \autocite[Conjecture 1.2]{AK}) predicts that all algebraic subvarieties $V \subseteq \mathbb{H}^n \times \mathbb{C}^n$ which satisfy freeness and broadness (which have more general definitions than the ones we use here) intersect the graph of $j$.

\begin{lem}\label{zaropen}
Let $L \times W \subseteq \mathbb{H}^n \times \mathbb{C}^n$ be a broad subvariety, where $L \subseteq \mathbb{H}^n$ is a M\"obius subvariety and $W \subseteq \mathbb{C}^n$ is algebraic. 

There is a non-empty Zariski open subset $W^\circ$ of $W$ such that for any analytic irreducible component $C$ of $(L \times W^\circ) \cap \Gamma_j$ we have $$\dim C=\dim L + \dim W-n$$ (i.e$.$, $C$ is a typical component of the intersection).
\end{lem}

\begin{proof}
For $I$ an ordered tuple of elements of $\{1,\dots, n\}$, let $\pi_I:\mathbb{C}^n \rightarrow \mathbb{C}^{n_I}$ denote the corresponding projection. Denote, for $w \in \pi_I(W)$, by $W_w$ the set $$\{w' \in W \mid \pi_I(w')=w\}.$$ By the fibre dimension theorem, for any $I$ there is a Zariski-closed proper subset $W_I \subseteq W$ such that for every $w \in W$, if $\dim (W_w) > \dim W - \dim \pi_I(W)$ then $w \in W_I$. 

Consider the family $$\{W_w \mid w \in \pi_I(W)\}$$ of subvarieties of $W$. By Aslanyan's weak modular Zilber-Pink theorem for parametric families \autocite[Theorem 7.9]{A20}, there is a Zariski-closed subset $S_I$ of $W$ such that any atypical intersection between some variety $W_w$ and a weakly special subvariety of $\mathbb{C}^n$ with no constant coordinates is contained in $S_I$. Thus, define $$W^\circ:=W \setminus \bigcup_{I \subseteq [n]} (S_I \cup W_I).$$

Suppose now that $U$ is a bounded open subset of $L$, that $C$ is an analytic irreducible component of the intersection $j(U) \cap W^\circ$, and that $$\dim C > \dim L + \dim W -n \geq 0.$$ 

By Pila and Tsimerman's Ax-Schanuel theorem for $j$ \autocite[Theorem 1.1]{PT} we have that $C$ must be contained in an atypical intersection between $W$ and a weakly special subvariety of $\mathbb{C}^n$. Since $C \subseteq W^\circ$, and $S_I \subseteq W\setminus W^\circ$, we have that $C$ has some constant coordinates. Let then $$I:=\{i \in \{1,\dots,n\} \mid \textnormal{the } i\textnormal{-th coordinate is constant on }C \}.$$

There is some $a \in \mathbb{H}^{n_I}$ such that $C \subseteq W_{j(a)}$. Denote by $L_a$ and $U_a$ the fibres over $a$ for the restrictions of the projection $\pi_I$ to $L$ and $U$, and let $I_0$ be the complement of the set $I$ in $[n]$. The projection $\pi_{I_0}$ has zero-dimensional fibres on $L_{a}$ and on $W_{j(a)}$, so it preserves dimensions. The component $\pi_{I_0}(C)$ of the intersection $$j(\pi_{I_0}(U_a)) \cap \pi_{I_0}(W_{j(a)})$$ then has no constant coordinates, and its points do not identically satisfy modular relations because they are projections of points in $C$, so it is a typical component of the intersection. Thus, 
\begin{align*}
    \dim L_a + \dim W_{j(a)} -(n-n_I)&=\dim \pi_{I_0}(L)+\dim \pi_{I_0}(W)-(n-n_I)\\
                                     &=\dim \pi_{I_0}(C) \\
                                     &=\dim C\\
                                     &> \dim L + \dim W -n.
\end{align*}  
Hence $$\dim W_{j(a)} > \dim L - \dim L_{a} + \dim W -n_I.$$

Since $C \subseteq W^\circ$ and $L \times W$ is broad, $$\dim W_{j(a)} = \dim W - \dim \pi_I(W) \leq \dim W - n_I +\dim \pi_I(L).$$

Comparing these, we obtain
\begin{align*}
    \dim W - n_I + \dim \pi_I(L) &> \dim L - \dim L_{a} + \dim W -n_I \\
    \dim L_{a} + \dim \pi_{I}(L) &> \dim L
\end{align*}
which cannot hold as $L$ is a M\"obius subvariety. Therefore the component $C$ is typical.
\end{proof}

We can now prove the main result. 

\begin{thm}\label{Main}
	Let $L\subseteq \mathbb{H}^n$ be a M\"obius subvariety, $W \subseteq \mathbb{C}^n$ an algebraic subvariety such that the subvariety $L \times W$ is free broad. Then $j(L) \cap W$ is dense in $W$ in the Euclidean topology.
\end{thm}

\begin{proof}
	By induction on $d=\dim L$. The case $d=1$ is Proposition \ref{d=1}, so suppose the theorem holds for $d$, $\dim L=d+1$, and $\dim W\geq n-d-1$. By intersecting $W$ with generic hyperplanes so that the dimension of the projections stays sufficiently big, we may assume without loss of generality that $\dim W=n-d-1$.
	
	By definition of M\"obius subvariety, up to reordering the coordinates we can write $L$ as a product $L_1 \times \dots \times L_{d+1}$, where each $L_i$ is a one-dimensional M\"obius subvariety. There are numbers $n_1$, $n_2$ such that $n_1+n_2=n$, $L':=L_1 \times \dots \times L_d$ is a $d$-dimensional M\"obius subvariety of $\mathbb{H}^{n_1}$ and $L_{d+1}$ is a 1-dimensional M\"obius subvariety of $\mathbb{H}^{n_2}$; let $\pi_{i}:\mathbb{C}^n \rightarrow\mathbb{C}^{n_{i}}$, for $i=1,2$ denote the corresponding projections on the codomain. By broadness, $\pi_1(W)$ has dimension at least $n_1-d$, and hence by the inductive hypothesis it contains a dense subset of points of $j(L')$; now there are two cases.
	
	If $\dim(\pi_{1}(W))=n_{1}-d$, then by the fibre dimension theorem any point $w_{1} \in \pi_1(W) \cap j(L')$ has a fibre $W_{w_1}$ of dimension at least $\dim W-(n_{1}-d)=n_2-1$. Therefore by Lemma \ref{holo} $\pi_2(W_{w_{1}})$ has a dense subset of points of $j(L_{d+1})$, and we are done.
	
	If $\dim (\pi_{1}(W))=n_{1}-d+k$ for some $k>0$, then for a generic point $w \in \pi_{1}(W)$ the fibre $W_{w}$ has dimension $n_2-1-k$.
	
	As the subvariety $L' \times \pi_{1}(W)$ is broad, by the inductive hypothesis and intersecting $\pi_{1}(W)$ with generic hyperplanes if necessary we may assume $L' \times \pi_1(W)$ intersects the graph of $j$ in an analytic set of dimension $k' \geq k$. Then, denoting by $\Gamma_j$ the graph of $j$, $$\pi_1^{-1}((L' \times \pi_1(W)) \cap \Gamma_j)=((L' \times \mathbb{C}^{n_2}) \times W) \cap \Gamma_j$$ has dimension $k'+n_2-1-k \geq n_2 -1$.
	
	Now let $U$ be a bounded open subset of $L'$, so that $(j(U) \times \mathbb{C}^{n_2}) \cap W$ is an analytic set in $j(U) \times \mathbb{C}^{n_2}$, and let $\pi_{\res}$ denote the restriction of the second projection $\pi_2$ to the set $(j(U) \times \mathbb{C}^{n_2}) \cap W$.
	
	Suppose that $\pi_{\res}$ has finite fibres: then it is proper. To prove this we use the following fact.
	
	\begin{fact}[\cite{Ch}, p.3.1]\label{propproj}
		Let $X$ and $Y$ be locally compact, Hausdorff topological spaces, $G \subseteq X$ and $D \subseteq Y$ subsets with $\o{G}$ compact. If $A$ is a closed subset of $G \times D$, the projection $\pi_2:A \rightarrow D$ is proper if and only if $A$ has no limit points in $\partial G \times D$.
	\end{fact}
	
	In our setting, $X=\mathbb{C}^{n_1}$, $Y=D=\mathbb{C}^{n_2}$, $G=j(U)$, and $A=j(U) \times \mathbb{C}^{n_2} \cap W$; let $w_2$ be a point in the image of $\pi_{\res}$ with finite fibre. Given that $\pi_{\res}$ has finite fibres, there is a ball $B \subseteq j(U)$ such that $W_{w_2} \cap j(U) \times \mathbb{C}^{n_2}$ does not intersect the set $\partial B \times \{w_2\}$. As $W \cap j(U) \times \mathbb{C}^{n_2}$ is closed in $j(U) \times \mathbb{C}^{n_2}$ and $\pi_{\res}$ is finite and hence open, this property actually holds in a neighbourhood of $w_2$; therefore, by Fact \ref{propproj}, by taking $U$ sufficiently small we can make sure that the map $\pi_\res$ is proper. 
	
	So, under the assumption that $\pi_{\res}$ has finite fibres, we can apply the proper mapping theorem (see p.5.8 in \cite{Ch}), which states that the image of $\pi_{\res}$ is an analytic set. Since we proved $j(U) \times \mathbb{C}^{n_2} \cap W$ has dimension at least $n_2-1$, the image of $\pi_{\res}$ is either an open subset of $\mathbb{C}^{n_2}$, or an analytic set in $\mathbb{C}^{n_2}$ of codimension $1$; either way, using density of $j(L_{d+1})$ in the first case and Lemma \ref{holo} otherwise (by freeness none of the coordinates is identically $0$ or $1728$), we can find a point $w \in W \cap j(U) \times \mathbb{C}^{n_2}$ such that $\pi_2(w) \in j(L_{d+1})$; therefore, $w \in j(L) \cap W$, as we wanted.
	
	Therefore, it remains to show that if we choose the bounded open subset $U$ of $L'$ appropriately, then the map $\pi_{\res}$ has finite fibres. Using once again broadness and the fibre dimension theorem, we find a point $w_2 \in \pi_2(W)$ with fibre $W_{w_2}:=\{w \in W \mid \pi_2(w)=w_2 \}$ such that 
    \begin{align*}
        \dim W_{w_2} + \dim (j(L') \times \mathbb{C}^{n_2} ) &= \dim W - \dim \pi_2(W) + \dim j'(L) + n_2 \\
        & \leq (n_1-d-1)-(n_2-1) +d+n_2 \\
        &= n.
    \end{align*}
    Therefore any positive dimensional intersection between $j(L') \times \mathbb{C}^{n_2}\times W_{w_2}$ and $\Gamma_j$ is atypical; by Lemma \ref{zaropen}, it suffices to make sure that $j(U) \times \mathbb{C}^{n_2} \cap W$ is contained in the Zariski-open non-empty subset $W^\circ$ of $W$ to avoid this. The set $W^\circ$ contains points of $j(L') \times \mathbb{C}^{n_2}$ because $\pi_{1}(W) \cap j(L')$ is dense in $\pi_{1}(W)$.
\end{proof}

\section{Derivatives of the $j$ Function}

We conclude with some remarks on extensions of the results in the previous section to the first derivative of the $j$ function. Recall that $j,j'$ and $j''$ are algebraically independent, and therefore many results in this area, for example, the Ax-Schanuel theorem \autocite[Theorem 1.2]{PT}, tend to consider them simultaneously. In turn, this leads to the formulation of questions about the existence of points of the form $(z,j(z),j'(z),j''(z))$ on algebraic subvarieties $V \subseteq \mathbb{H}^n \times \mathbb{C}^{3n}$ which satisfy versions of freeness and broadness (see \autocite[Conjecture 1.6]{AK}).

The methods in this paper seem to be insufficient to address the problems of systems of equations which involve $j,j'$ and $j''$; however, they can be employed to obtain some partial results on $j'$.

First of all we remark that while $j$ is a modular function, $j'$ and $j''$ are not: $j'$ is a modular form of weight 2, and the transformation law for $j''$ under the action of $\sl_2(\mathbb{Z})$ is more complicated. 

A partial solution to that is to work in jet spaces, as the second jet of the $j$ function is indeed invariant under the action of $\sl_2(\mathbb{Z})$ on $J_2\mathbb{H}$. However, this leads to other issues: the action of $\sl_2(\mathbb{R})$ on $J_2\mathbb{H}$ is not transitive, and points in the image of $J_2j$ do not have the form $(j(z), j'(z), j''(z))$, but rather $(j(z), j'(z)r, j''(z)r^2+j'(z)s)$ for some complex numbers $r$ and $s$. 

In this section we show one kind of result that can still be obtained, namely that a free hypersurface in the complex numbers intersects the image of the first jet of M\"obius subvariety of $\mathbb{H}^n$ under the first jet of $j$. 

We note that as in the previous sections we only consider irreducible algebraic varieties.

\subsection{Background and Notation}

We recall some general facts about jet spaces, and about the action of $\sl_2(\mathbb{R})$ on $J_2\mathbb{H}$.

\begin{defn}
Let $M$ be a complex analytic manifold. The \textit{$k$-th jet space} of $M$ for a natural number $k$ is the space of equivalence classes of holomorphic maps $f:U \rightarrow M$ from a small neighbourhood $U$ of $0 \in \mathbb{C}$ into $M$, identifying maps that are equal up to order $k$.
\end{defn}

We will only be interested in second jets, so we assume $k=2$ in the following.

An element in $J_2\mathbb{H}$ is a triple $(z,r,s)$, where $z \in \mathbb{H}$, $r,s \in \mathbb{C}$, that corresponds to the equivalence class containing the function $f:U \rightarrow \mathbb{H}$ taking $w$ to $z+rw+s\frac{w^2}{2}$, where $U$ is a neighbourhood of 0. 

Jets are a functorial construction: given a map $\phi:M \rightarrow N$, there is an induced map $J_k\phi:J_kM \rightarrow J_kN$, that takes the equivalence class of the function $f:U \rightarrow M$ to that of $\phi \circ f$. Therefore, if for a fixed $g=\mat{a}{b}{c}{d} \in \sl_2(\mathbb{R})$ we consider the map $g \cdot (-)  :\mathbb{H} \rightarrow \mathbb{H}$, we can see what the action induced on $J_2\mathbb{H}$ is: $$g \cdot (z,r,s)=\left( \frac{az+b}{cz+d}, \frac{r}{(cz+d)^2}, \frac{s}{(cz+d)^2} - \frac{2cr^2}{(cz+d)^3} \right)$$

Similarly we can consider the second jet of the $j$ function itself, which is obtained as:

$$J_2j(z,r_1,r_2)=\left(j(z), j'(z)r, j''(z)r^2+j'(z)s \right)$$ so that in particular for example $J_2j(z,1,0)=(j(z),j'(z),j''(z))$ for any $z \in \mathbb{H}$.

Using the transformation laws for $j'$ and $j''$, i.e$.$, for $\gamma=\mat{a}{b}{c}{d}$, $$j'(\gamma z)=(cz+d)^2 j'(z)$$ and $$j''(\gamma z)=(cz+d)^4j''(z)+2c(cz+d)^3j'(z)$$ one can prove that $J_2j(\gamma \cdot (z,r,s))=J_2j(z,r,s)$.

\subsection{Intersections for $j'$}

Let $T_1j:J_1\mathbb{H}^n \rightarrow \mathbb{C}^n$ denote the composition $\pi \circ J_1 j$, where $\pi:J_1 \mathbb{C}^n \cong \mathbb{C}^{2n} \rightarrow \mathbb{C}^2$ is the projection on the third and fourth coordinate. Hence, $$T_1(j)(z_1,z_2,r_1,r_2)=\left(j'(z_1)r_1, j'(z_2)r_2 \right).$$

In this subsection we prove the following statement, a $j'$-algebraic-closedness type result which can be obtained with the methods of this paper.

\begin{thm}\label{Main'}
    Let $L \subseteq \mathbb{H}^n$ be a free M\"obius subvariety of dimension 1, $W \subseteq \mathbb{C}^n$ a free hypersurface.

    Then $T_1j(J_1L) \cap W$ is dense in $W$ in the Euclidean topology.
\end{thm}

\begin{lem}\label{onedim}
Let $(w_1,w_2) \in \mathbb{C}^2$, both non-zero. Then there exist $z_1 \in \mathbb{H}$ and $h \in \sl_2(\mathbb{R})$, $h \notin \gl_2(\mathbb{Q})^+$, such that $\frac{j'(hz_1)}{(cz_1+d)^2}=w_2$.
\end{lem}

\begin{proof}
Fix $(w_1,w_2) \in \mathbb{C}^2$ and $z_1$ such that $j'(z_1)=w_1$ (which exists by surjectivity of $j'$). Using the Fourier series expansion of $j$, $$j(z)=\sum_{n=-1}^\infty c_ne^{2\pi niz} $$ where all $c_n$'s are positive integers (see for example \cite{Rad}) we see that $j'$ takes imaginary values on $i\mathbb{R}_{\geq 1}$ and that $$\lim_{x \in \mathbb{R},x \rightarrow +\infty} |j'(ix)|=\infty.$$

Consider the function $\phi:\mathbb{H} \rightarrow \mathbb{C}$ defined by $z  \mapsto j'(z) \Im(z)$. A direct computation shows that the Jacobian of $\phi$ as a function $\mathbb{R} \times \mathbb{R}_{>0} \rightarrow \mathbb{R}^2$ is nonsingular on points of $\{0\} \times \mathbb{R}_{\geq 1}$. Hence, as $j'(i)=0$, for every $x \in \mathbb{R}_{\geq 0}$ there is a $z \in i\mathbb{R}_{\geq 1}$ such that $|\phi(z)|=x$ and $\phi$ is open around $z$.

Hence we may find some $y \in \mathbb{R}_{\geq 0}$ such that $|j'(iy)\Im(iy)|=|w_2\Im(z_1)|$. As $z \mapsto |\phi(z)|$ is a real analytic mapping from a space of dimension 2 into a space of dimension 1, the set $$\{z \in \mathbb{H} \mid |\phi(z)|=|w_2\Im(z_1)| \}$$ contains a real analytic 1-dimensional neighbourhood of $iy$, and so in particular a point $z \notin \gl_2(\mathbb{Q})^+ \cdot z_1$. Let $h=\mat{a}{b}{c}{d} \in \sl_2(\mathbb{R})$ satisfy $hz_1=z$. Then we have 
\begin{align*}
    |j'(z)\Im(z)|&=|w_2 \Im(z_1)| \\
    |j'(hz_1) \Im(hz_1)| &= |w_2 \Im(z_1)|\\ 
    \left| \frac{\Im(hz_1)}{\Im(z_1)} \right| &= \left| \frac{w_2}{j'(hz_1)} \right|
\end{align*}
It is easy to check that $\left| \frac{\Im(hz_1)}{\Im(z_1)} \right| =\frac{1}{|cz+d|^2}$, hence
\begin{align*}
    \left| \frac{w_2}{j'(hz_1)} \right| &=  \frac{1}{|cz_1+d|^2}.
\end{align*}
By Lemma \ref{sincos}, we can choose $h$ so that 
\begin{align*}
    \frac{w_2}{j'(hz_1)} &=  \frac{1}{(cz_1+d)^2}.
\end{align*} 
\end{proof}

\begin{prop}\label{zardensej'}
    Let $L \subseteq \mathbb{H}^n$ be a free M\"obius subvariety of dimension 1. Then $T_1j(J_1L)$ is Zariski-dense in $\mathbb{C}^n$.
\end{prop}

\begin{proof}
    Let $W \subseteq \mathbb{C}^n$ be a hypersurface. If $T_1j(J_1) \subseteq W$, then $J_1j(J_1L) \subseteq \mathbb{C}^n$, that is, the intersection of the variety $J_1L \times (\mathbb{C}^n \times W)$ with the graph of $J_1j$ has dimension 2. This is an atypical intersection and so by the Ax-Schanuel theorem with derivatives the projection to the first $n$ coordinates of the codomain needs to be contained in a weakly special subvariety of $\mathbb{C}^n$, contradicting freeness of $L$.
\end{proof}

We recall the \textit{cocycle relation for automorphy factors}, as it will be used in the proof: if $g=\mat{a_g}{b_g}{c_g}{d_g}, h=\mat{a_h}{b_h}{c_h}{d_h} \in \sl_2(\mathbb{R})$, and $gh=\mat{a_{gh}}{b_{gh}}{c_{gh}}{d_{gh}}$, then for all $z \in \mathbb{H}$ we have $$c_{gh}z+d_{gh}=(c_ghz+d_g)(c_hz+d_h).$$ This can be easily verified directly.

\begin{proof}[Proof of Theorem \ref{Main'}]
    For notational simplicity we deal with the case $n=2$; the general case is completely analogous. In this proof we adopt the following convention: if we denote a matrix in $\sl_2(\mathbb{R})$ by a letter, say $h$, we use the notation $\mat{a_h}{b_h}{c_h}{d_h}$ for its entries.

Let $f \in \mathbb{C}[X_1,X_2]$ be the polynomial defining $W$, and let $(w_1, w_2) \in W$ be a regular point. By Lemma \ref{onedim}, there are $z_1 \in \mathbb{H}$ and $h \in \sl_2(\mathbb{R})$, $h \notin \gl_2(\mathbb{Q})^+$, such that $j'(z_1)=w_1$ and $\frac{j'(hgz_1)}{(c_{hg}z+d_{hg})^2}=w_2$, so we have that $$h \cdot \left( z_1, 1 \right) = \left( hz_1, \frac{1}{(cz_1+d)^2} \right) $$ and $$Tj_1\left(z_1, hz_1, 1, \frac{1}{(cz_1+d)^2} \right)= (w_1,w_2).$$ Hence, if we denote by $L_h$ the M\"obius subvariety of $\mathbb{H}^2$ defined as $$L_h:=\{(z,hz) \in \mathbb{H}^2 \mid z \in \mathbb{H} \}$$ we get that $(w_1,w_2) \in W \cap T_1j(J_1L_h)$. Let $U$ be a neighbourhood of $z_1$ in $\mathbb{H}$, and let $F:U \rightarrow \mathbb{C}$ denote the function defined by $$z \mapsto f\left( T_1j \left(z,ghz, 1, \frac{1}{(c_{gh}z+d_{gh})^2} \right) \right)=f\left( j'(z), \frac{j'(hgz)}{(c_{hg}z+d_{hg})^2} \right).$$ By Proposition \ref{zardensej'}, $F$ is a non-constant holomorphic function, so $\im(F)$ is an open neighbourhood of 0.

By Lemma \ref{densityinsl2} there is a sequence $\{(\gamma_1^i  k_i , \gamma_2^i g k_i g^{-1})\}_{i \in \mathbb{N}}$, where each $\gamma_1^i, \gamma_2^i \in \sl_2(\mathbb{Z})$ and each $k_i \in \sl_2(\mathbb{R})$, which converges to $(\mathbb{I}_2, h).$ Consider the sequence of functions $\{F_i\}_{i \in \mathbb{N}}$, where each $F_i:U \rightarrow \mathbb{C}$ is defined by 
\begin{align*}
    z &\mapsto  f \left( T_1 j \left( \gamma_1^i k_i z, \gamma_2^i g k_i z, \frac{1}{(c_{\gamma_1^i k_i}z+d_{\gamma_1^i k_i})^2}, \frac{1}{(c_{\gamma_2^i gk_i}z+d_{\gamma_2^i gk_i})^2}  \right) \right)\\
    &= f\left( \frac{j'(\gamma_1^i k_i z)}{(c_{\gamma_1^i k_i} z+d_{\gamma_1^i k_i})^2}, \frac{j'(\gamma_2^i gk_i z)}{(c_{\gamma_2^i gk_i} z+d_{\gamma_2^i gk_i})^2} \right).
\end{align*}
The $F_i$'s converge uniformly to $F$. As 0 is in the interior of the image of $F$, for sufficiently large $i$ we have that $0 \in \im(F_i)$, so there is $z \in U$ such that $$f\left( \frac{j'(\gamma_1^i k_i z)}{(c_{\gamma_1^i k_i} z+d_{\gamma_1^i k_i})^2}, \frac{j'(\gamma_2^i gk_i z)}{(c_{\gamma_2^i gk_i} z+d_{\gamma_2^i gk_i})^2} \right)=0.$$

By the cocycle relation we have that for each $i \in \mathbb{N}$, $$\left( \gamma_1^i k_iz, \gamma_2^i gk_i z, \frac{1}{(c_{\gamma_1^i k_i}z + d_{\gamma_1^i k_i})^2}, \frac{1}{(c_{\gamma_2^ig k_i}z + d_{\gamma_2^ig k_i})^2} \right)=$$ $$=(\gamma_1^i, \gamma_2^i) \cdot \left( k_i z, gk_iz,\frac{1}{(c_{k_i}z+d_{k_i})^2}, \frac{1}{(c_{gk_i}z+d_{gk_i})^2} \right)$$ and therefore, by $\sl_2(\mathbb{Z})$-invariance of $J_1j$ (and hence of $T_1j$), $$f \left(T_1j \left( k_i z, gk_iz,\frac{1}{(c_{k_i}z+d_{k_i})^2}, \frac{1}{(c_{gk_i}z+d_{gk_i})^2}\right) \right)=0.$$  Hence, $((j(k_iz))', (j(gk_iz))')=\left( \frac{j'(k_iz)}{(c_{k_i}z_1+d_{gk_i})^2}, \frac{j'(gk_iz)}{(c_{gk_i}z_1+d_{gki})^2} \right) \in W \cap T_1j(J_1L)$.
\end{proof}

\printbibliography

\end{document}